\documentclass[10pt,a4paper,oneside,reqno]{amsart}

\usepackage{setspace}
\usepackage[totalwidth=14cm,totalheight=22cm]{geometry}
\usepackage[T1]{fontenc}
\usepackage[dvips]{graphicx}
\usepackage{epsfig}
\usepackage{amsmath,amssymb,amscd,amsthm}
\usepackage{subfigure}
\usepackage[latin1]{inputenc}
\usepackage{latexsym}
\usepackage{graphics}
\usepackage{colortbl} 
\usepackage{color}
\usepackage{ae}
\usepackage{enumitem}
\usepackage[all]{xy}
\usepackage{scalerel}
\usepackage{tikz}
\usepackage{cancel}
\usepackage{bbm,amsbsy}
\usepackage{mathrsfs}  
\usepackage[colorlinks=true,pagebackref=true]{hyperref}
\usepackage[normalem]{ulem}
\usepackage{nicefrac}
\usepackage{xfrac}
\usepackage{faktor}
\usepackage{tikz-cd} 

\usepackage{nomencl}
\usepackage[normalem]{ulem}

\makenomenclature

\makeatletter
\newtheorem*{rep@theorem}{\rep@title}
\newcommand{\newreptheorem}[2]{%
\newenvironment{rep#1}[1]{%
 \def\rep@title{#2 \ref{##1}}%
 \begin{rep@theorem}}%
 {\end{rep@theorem}}}
\makeatother

\makeatletter
\newtheorem*{rep@cor}{\rep@title}
\newcommand{\newrepcor}[2]{%
\newenvironment{rep#1}[1]{%
 \def\rep@title{#2 \ref{##1}}%
 \begin{rep@cor}}%
 {\end{rep@cor}}}
\makeatother

\makeatletter
\newtheorem*{rep@prop}{\rep@title}
\newcommand{\newrepprop}[2]{%
\newenvironment{rep#1}[1]{%
 \def\rep@title{#2 \ref{##1}}%
 \begin{rep@prop}}%
 {\end{rep@prop}}}
\makeatother

\newtheorem{cor}{Corollary}[section]
\newtheorem{corx}{Corollary}

\newtheorem{theorem}[cor]{Theorem}
\newtheorem{thmx}[corx]{Theorem}

\newtheorem{prop}[cor]{Proposition}

\newrepcor{cor}{Corollary}
\newreptheorem{theorem}{Theorem}
\newrepprop{prop}{Proposition}

\newtheorem{lemma}[cor]{Lemma}

\theoremstyle{definition}
\newtheorem{defi}[cor]{Definition}
\theoremstyle{remark}
\newtheorem{remark}[cor]{Remark}
\newtheorem*{remark*}{Remark}

\newtheorem*{notation*}{Notation}

\newlist{steps}{enumerate}{1}
\setlist[steps, 1]{itemsep=8pt,leftmargin=0cm,itemindent=.5cm,labelwidth=\itemindent,labelsep=0cm,align=left,label = \textbf{\emph{Step \arabic*}:\,}}

\makeatletter
\newcommand{\myitem}[1]{%
\item[#1]\protected@edef\@currentlabel{#1}%
}
\makeatother

\newcommand{\QF}{{\mathcal{QF}}}
\newcommand{\Q}{{\mathcal{Q}}}
\newcommand{\F}{{\mathcal{F}}}
\newcommand{\D}{{\mathbb D}}
\newcommand{\C}{{\mathbb C}}
\newcommand{\R}{{\mathbb R}}

\newcommand{\Hyp}{\mathbb{H}}

\newcommand{\PSL}{\mathrm{PSL}}

\newcommand{\Eps}{\mathrm{Eps}}
\newcommand{\CP}{\C\mathrm{P}}

\newcommand{\T}{\mathcal{T}(\Sigma)}

\begin{document}

\setcounter{secnumdepth}{3}
\setcounter{tocdepth}{2}

\title[Quasi-Fuchsian manifolds  foliated by CMC surfaces]{Quasi-Fuchsian manifolds close to the Fuchsian locus are foliated by constant mean curvature surfaces}

\author[Diptaishik Choudhury]{Diptaishik Choudhury}
\address{Diptaishik Choudhury: Department of Mathematics, University of Luxembourg, 6 avenue de la Fonte, L-4364 Esch-Sur-Alzette, Luxembourg.} \email{diptaishik.choudhury@uni.lu }

\author[Filippo Mazzoli]{Filippo Mazzoli}
\address{Filippo Mazzoli: Department of Mathematics, University of Virginia, 
	141 Cabell Drive, 
	22904-4137 Charlottesville VA, USA.} \email{filippomazzoli@me.com}

\author[Andrea Seppi]{Andrea Seppi}
\address{Andrea Seppi: Institut Fourier, UMR 5582, Laboratoire de Math\'ematiques,
Universit\'e Grenoble Alpes, CS 40700, 38058 Grenoble cedex 9, France.} \email{andrea.seppi@univ-grenoble-alpes.fr}

\date{\today}

\thanks{The third author is a member of the national research group GNSAGA}

\maketitle

\begin{abstract}
Even though it is known that there exist quasi-Fuchsian hyperbolic three-manifolds that do not admit any monotone foliation by constant mean curvature (CMC) surfaces, a conjecture due to Thurston asserts the existence of CMC foliations for all \emph{almost-Fuchsian} manifolds, namely those quasi-Fuchsian manifolds that contain a closed minimal surface with principal curvatures in $(-1,1)$. In this paper we prove that there exists a (unique) monotone CMC foliation for all quasi-Fuchsian manifolds that lie in a sufficiently small neighborhood of the Fuchsian locus.
\end{abstract}

\tableofcontents

\section{Introduction}

\emph{Quasi-Fuchsian manifolds} are an important class of complete hyperbolic three-manifolds, largely studied in geometric topology. If $\Sigma$ a closed orientable surface of genus $g \geq 2$, the deformation space of quasi-Fuchsian manifolds homeomorphic to $\Sigma\times\R$, which we denote here by $\QF(\Sigma)$, is known to be homeomorphic to the product of two copies of the Teichm\"uller space $\T$ of $\Sigma$ by Bers' double uniformization theorem, and is therefore a manifold of real dimension $12g-12$. Those quasi-Fuchsian manifolds that contain a closed totally geodesic surface are called \emph{Fuchsian}, and they bijectively correspond to points of the diagonal of $\T\times\T$ through Bers' parameterization.

The investigation of quasi-Fuchsian manifolds has been tackled both with a combinatorial flavour, for instance through the notion of \emph{pleating laminations} \cite{thurston1979geometry, sullivan1981travaux,Bo98variation,zbMATH04098303,Bo96shearing,zbMATH02211626,zbMATH02201103,EM06,zbMATH05056497,zbMATH05027326,Bo98schlafli,   filippo,dip}  and with a more analytic approach, for example using \emph{minimal surfaces} \cite{zbMATH03840752,zbMATH02227015,zbMATH05200423,zbMATH05279031,zbMATH05851259,zbMATH06204974,seppiminimal,zbMATH06653236,zbMATH07237228,huanglowe}. In this paper we continue the analytic study of quasi-Fuchsian manifolds, and in particular of \emph{foliations} whose leaves are surfaces of \emph{constant mean curvature} (CMC), as in the following definition:

\begin{defi}\label{defi foliation}
A Riemannian three-manifold  $M$ homeomorphic to $\Sigma\times\R$ is (smoothly) monotonically foliated by CMC surfaces with mean curvature ranging in the interval $(a,b)$ if there exists a diffeomorphism between $\Sigma\times(a,b)$ and $M$ which, for every $H\in (a,b)$, is an embedding of constant mean curvature $H$ when restricted to $\Sigma\times\{H\}$.
\end{defi}

\vspace{0.1cm}

\subsection*{Historical pespective} Let us give a brief historical perspective to the problem. It is known that there exist quasi-Fuchsian manifolds containing several closed minimal surfaces homotopic to $\Sigma\times\{*\}$, see \cite{zbMATH03876145} and \cite{zbMATH06460565}. In particular, this implies that there exists quasi-Fuchsian manifolds $M$ that do not admit a global monotone CMC foliation. Indeed if $M\cong \Sigma\times\R$ admits a monotone CMC foliation (as in Definition \ref{defi foliation}), then by a simple application of the geometric maximum principle, the closed embedded minimal surface in $M$ homotopic to $\Sigma\times\{*\}$ would be unique. 

Concerning uniqueness of minimal surfaces, the work of Uhlenbeck \cite{zbMATH03840752} highlighted the importance of a class of quasi-Fuchsian manifolds, which has been later called \emph{almost-Fuchsian} in \cite{zbMATH05200423}, defined by the existence of a closed minimal surface with principal curvatures in $(-1,1)$. This condition actually implies that the minimal surface is unique, and that the equidistant surfaces from the minimal surface provide a global foliation of $M$. However, the leaves of this equidistant foliation do not have constant mean curvature, except in the trivial case where $M$ is {Fuchsian}.

Thurston conjectured that every almost-Fuchsian manifold is foliated by CMC surfaces. However, to the best of our knowledge, Fuchsian manifolds are so far the only known examples of quasi-Fuchsian manifolds that are (monotonically) foliated by CMC surfaces.

Before stating our result, let us turn our attention to some positive results in this direction. By a special case of the results of Mazzeo and Pacard in \cite{MP}, each end of any quasi-Fuchsian manifold (namely, each connected component of the complement of a compact set homeomorphic to $\Sigma\times I$ for $I$ a closed interval) is smoothly monotonically foliated by CMC surfaces, with mean curvature ranging in $(-1,-1+\epsilon)$ and  $(1-\epsilon,1)$. This result has been reproved by Quinn in \cite{quinn}, using an alternative approach which is extremely relevant for the present work. Moreover, the recent work of Guaraco-Lima-Pallete \cite{GLP} showed that every quasi-Fuchsian manifold admits a global foliation in which every leaf has constant sign of the mean curvature, meaning that it is either minimal or the mean curvature is nowhere vanishing on the entire leaf.

We also remark that existence results for CMC surfaces in the hyperbolic three-space with a given boundary curve at infinity, and in quasi-Fuchsian manifolds, have been obtained in \cite{zbMATH06553438,zbMATH06759193,zbMATH06993267}. These results rely on delicate applications of geometric measure theory methods.

\clearpage

\subsection*{Main result and outline of the strategy}

Let us now state the main result of this paper. 

\begin{thmx}\label{thm:foliation}
	Let $\Sigma$ be a closed oriented surface of genus $\geq 2$. Then there exists a neighbourhood $U$ of the Fuchsian locus in quasi-Fuchsian space $\QF(\Sigma)$ such that every quasi-Fuchsian manifold in $U$ is smoothly monotonically foliated by CMC surfaces, with mean curvature ranging in $(-1,1)$.
\end{thmx}

The monotone CMC foliation of a quasi-Fuchsian manifold $M\cong\Sigma\times\R$, when it exists, is automatically unique by a standard application of the geometric maximum principle. More precisely, the leaf of the foliation with mean curvature $H$ is the unique closed surface homotopic to $\Sigma\times\{*\}$ in $M$ having mean curvature identically equal to $H$.

Observe that, if a quasi-Fuchsian manifold admits a monotone CMC foliation, then the mean curvature necessarily ranges in $(-1,1)$. Indeed, any leaf of the foliation must necessarily have mean curvature in $(-1,1)$, see \cite[Lemma 2.2]{zbMATH05046844}. Moreover, by the aforementioned result of Mazzeo-Pacard, the mean curvature converges to $-1$ and $1$ as we approach the ends.
 
We remark that the methods of our proof, which we outline below, also provide a direct proof of the \emph{existence} of closed embedded CMC surfaces of  
 mean curvature $H\in (-1,1)$ in the quasi-Fuchsian manifolds $M$ within the neighbourhood $U$. (See Theorem \ref{thm:existence}.) Our proof is independent of previous result in the literature, and does not rely on geometric measure theory techniques.

The rough idea of the proof of Theorem \ref{thm:foliation} is to ``combine'' foliations of the ends, which have been provided in the works of Mazzeo-Pacard and Quinn for every quasi-Fuchsian manifold, with foliations of the compact part that we obtain by a ``deformation'' from Fuchsian manifolds. For the foliations of the ends, we adapt the proof given by Quinn in \cite{quinn}, which relies on the Epstein map construction (\cite{epstein,Dumas2017}), that associates to a conformal metric defined  in (a subset of) the boundary at infinity of $\Hyp^3$ an immersed surface in $\Hyp^3$ by ``envelope of horospheres''. One can then translate the condition of constant mean curvature into a PDE on the conformal factor, to which we apply an implicit function theorem method in an infinite-dimensional setting. The fact that the obtained solutions provide a smooth monotone foliation of the complement of a large compact set in the quasi-Fuchsian manifold $M$ follows from another application of the implicit function theorem. The main difference with respect to Quinn's proof is that we refine his method in order to achieve the existence of monotone foliations by CMC surfaces of mean curvature $(-1,-1+\epsilon)\cup (1-\epsilon,1)$ for \emph{any} quasi-Fuchsian manifold in a neighbourhood $U_M$ of a given  $M\in\QF(\Sigma)$, where the constant $\epsilon$ is \emph{uniform} over $U_M$ (Theorem \ref{thm:foliation_ends}).

For the compact part, we again obtain the existence of CMC surfaces, for $H\in (-1,1)$, with an implicit function theorem method in infinite-dimensional spaces, using the Epstein construction. In this case, however, the initial solution to which we apply the implicit function theorem is not ``at infinity''; it is instead the umbilical CMC surface in a Fuchsian manifold. In other words, we ``deform'' CMC surfaces in a Fuchsian manifold $M'$ to nearby quasi-Fuchsian manifolds in a neighbourhood $U_{M'}$. Similarly as above, the main technical difficulty is to have a uniform control of the constants, which must not depend on the quasi-Fuchsian manifold as long as we remain in the neighbourhood $U_{M'}$. See Theorem \ref{thm:existence_compact}.

The proof of Theorem \ref{thm:foliation} is then concluded by showing that these surfaces patch together to  a global smooth monotone foliation (Section \ref{sec:finish}), by means of a combination of a careful analysis of the constructed open sets in $\QF(\Sigma)$ and of several geometric arguments, for instance applications of the geometric maximum principle, relying on the observation that the CMC surfaces obtained as deformations from the Fuchsian locus can be assumed, up to restricting to smaller neighbourhoods, to have principal curvatures in $(-1,1)$.

\subsection*{Acknowledgements}
The authors are indebted to Zeno Huang for many discussions related to CMC surfaces and quasi-Fuchsian manifolds and for useful comments on a previous version of this manuscript. The authors are grateful to Jean-Marc Schlenker for useful discussions and for his encouragement. Finally, the second author would like to thank Gennady Uraltsev, for helpful conversations on some of analytic aspects of this work.

\section{Preliminaries}

\subsection{Quasi-Fuchsian manifolds}
 Given a discrete subgroup $\Gamma$  of $\PSL_{2}(\mathbb{C})$, the limit set $\Lambda_{\Gamma}$ of $\Gamma$ is defined to be the set of accumulation points in the boundary at infinity $\partial_\infty\mathbb H^3$ of the orbits of $\Gamma$ in $\mathbb H^3$. An oriented hyperbolic three-manifold 
$M$ is said to be a \emph{quasi-Fuchsian manifold} if it is isometric to the quotient $\mathbb{H}^{3}/\Gamma$, where $\Gamma\cong\pi_1(M)$ is a discrete subgroup of $\PSL_{2}(\mathbb{C})$ acting freely on $\Hyp^3$ and such that the limit set $\Lambda_{\Gamma}$ is a quasi-circle in  $\partial_{\infty}\mathbb{H}^{3}$. We will also call $\Gamma$ a \emph{quasi-Fuchsian group}, and the hyperbolic metric induced in the quotient $\Hyp^3/\Gamma$ by the metric of $\Hyp^3$ a \emph{quasi-Fuchsian metric}. In this case, $M$ is homeomorphic to $\Sigma \times \mathbb{R}$, where $\Sigma$ is an oriented surface.

\subsubsection{Deformation spaces} 
Let us now fix   \emph{closed} oriented surface $\Sigma$ of genus $g \geq 2$.
The deformation space of quasi-Fuchsian manifolds  is defined as:
$$\QF(\Sigma)=\{\text{quasi-Fuchsian metrics on }\Sigma\times\R \}/\mathrm{Diff}_0(\Sigma\times\R)~,$$
where the identity component $\mathrm{Diff}_0(\Sigma\times\R)$ of the diffeomorphism group acts by pull-back on the space of quasi-Fuchsian metrics.

An important subset of $\QF(\Sigma)$ is the \emph{Fuchsian locus}, which we denote by $\F(\Sigma)$. It consists of those quasi-Fuchsian manifolds which are isometric to $\Hyp^3/\Gamma$, such that the limit set $\Lambda_\Gamma$  coincides with the boundary of a totally geodesic plane $P$ in $\Hyp^3$. It follows that $\Gamma$ leaves invariant the totally geodesic plane $P$ whose boundary is $\Lambda_\Gamma$, and therefore the quotient $\Hyp^3/\Gamma$ contains a totally geodesic  surface homeomorphic to $\Sigma$. The induced metric on the totally geodesic surface can thus be considered as a point of the space of hyperbolic metrics on $\Sigma$, which will be denoted by $\mathcal M_{-1}(\Sigma)$. By means of this construction $\F(\Sigma)$ has a natural identification with the \emph{Teichm{\"u}ller space} of $\Sigma$, which we define here as
$$\T=\mathcal M_{-1}(\Sigma)/\mathrm{Diff}_0(\Sigma)~,$$
where  $\mathrm{Diff}_0(\Sigma\times\R)$ acts again by pull-back on $\mathcal M_{-1}(\Sigma)$.

\subsubsection{Conformal compactification}
If $\Gamma$ is a quasi-Fuchsian group, then the maximal domain of discontinuity of  the action of $\Gamma$ on $\partial_{\infty}\mathbb{H}^{3}$ is the complement of the limit set $\Lambda_\Gamma$. Since $\Lambda_\Gamma$ is a quasi-circle, hence in particular a Jordan curve, the domain of discontinuity has two connected components, which we denote by $\Omega^+$ and $\Omega^-$, both homeomorphic to an open disc. 

Observe that the orientations of $M$ and of $\Sigma$ allow to canonically distinguish $\Omega^+$ and $\Omega^-$. Indeed, the orientation of $\Sigma$ induces an orientation on its universal cover, which is homeomorphic to an open disc and, picking a lift of $\Sigma$ inside $\Hyp^3$, can be compactified to a closed disc by adding the limit set $\Lambda_\Gamma$. The curve $\Lambda_\Gamma$ then naturally comes with an orientation, and thus we can canonically label $\Omega^-$ the connected component on the left with respect to the orientation of $\Lambda_\Gamma$ and with the respect to the inward pointing normal vector to $\partial_\infty\Hyp^3$, and $\Omega^+$ the one on the right. With this convention,  it turns out that $(\Hyp^3\cup\Omega^+\cup\Omega^-)/\Gamma$ is orientation-preservingly homeomorphic to $\Sigma\times I$, for $I=[-1,1]$ a closed interval, in such a way that $\Omega^\pm/\Gamma$ is mapped homeomorphically to $\Sigma\times\{\pm1\}$. 

\begin{remark}\label{rmk:convention}
From now on, when $S$ is an embedded surface in $\Hyp^3/\Gamma$ homotopic to $\Sigma\times\{*\}$, we will refer to the \emph{(unit) normal vector} to $S$ as the one chosen according to the following convention. 
We lift $S$ to a surface $\widetilde S$ in the universal cover $\Hyp^3$, whose asymptotic boundary is the limit set $\Lambda_\Gamma$. Then $\widetilde S$ disconnects $\Hyp^3$ in two components. We declare that the unit normal vector to $S$ lifts to the unit normal vector to $\widetilde S$ pointing towards the component whose closure contains $\Omega^+$.


\end{remark}

 Since $\Gamma$ acts by biholomorphisms on the subsets $\Omega^\pm$ of $\partial_{\infty}\mathbb{H}^{3}\cong\C\mathrm P^1$, the above construction endows $\Sigma$ with two Riemann surface structures which we denote by $X^\pm$. We will  denote $(\Sigma,X^\pm)$ by $\partial_\infty^\pm M$, which allows us to write $\overline M=M\cup \partial_\infty^+ M\cup \partial_\infty^- M$, the \emph{conformal compactification} of $M$. By the celebrated Bers' double uniformization theorem, the pair of Riemann surface structures $(X^+,X^-)$ uniquely determines $M$ in $\QF(\Sigma)$. However, since we have defined the Teichm\"uller space as a quotient of the space $\mathcal M_{-1}(\Sigma)$ of hyperbolic metrics, we will instead consider the pair $(h^+,h^-)$, where $h^\pm$ is the unique hyperbolic metric on $\Sigma$ in the conformal class $X^\pm$, whose existence is guaranteed by the uniformization theorem. Hence the Bers' double uniformization theorem provides a homeomorphism between $\QF(\Sigma)$ and $\T\times\T$. However, for our purposes it will be more convenient to work with another (at least local) parameterization of $\QF(\Sigma)$, which we now describe.

\subsubsection{Schwarzian derivatives}\label{subsec:schwarzian}

Given a connected open set $\Omega \subset \C$ and a locally injective holomorphic map $f:\Omega\rightarrow \mathbb{C}$, the \emph{Schwarzian derivative} of $f$ is a holomorphic quadratic differential on $\Omega$ defined as: 
\begin{align*}
	S(f)= \left( \left(\frac{f''}{f'}\right)'-\frac{1}{2}\left(\frac{f''}{f'}\right)^{2} \right) dz^{2}
\end{align*}
This definiton has two fundamental properties: 
\begin{itemize}
	\item Given two locally injective holomorphic maps $f: \Omega \rightarrow \mathbb{C}$ and $g:f(\Omega)\to\C$,  
	\begin{equation}\label{eq:schw}
	S(g\circ f)= f^{*}S(g)+S(f)~.
	\end{equation}
	\item $S(f)=0$ if and only if $f$ is the restriction of an element of $\PSL(2,\mathbb{C})$.
\end{itemize}
In particular, 
if $\zeta$ is an element of $\PSL(2,\mathbb{C})$, then $S(\zeta\circ f)=S(f)$; conversely, if $f,g: \Omega \rightarrow \mathbb{C}$ are such that $S(f)=S(g)$ then $g=\zeta\circ f $ for some $\zeta$ in $\PSL(2,\mathbb{C})$.

Given a quasi-Fuchsian group $\Gamma$, consider the biholomorphic map $f:\D\to\Omega^+$ given by the Riemann mapping theorem, where $\D$ is now the unit disc. (Of course we could perform the same construction with $\Omega^-$, but we will focus on $\Omega^+$ for the sake of definiteness.) Conjugating the action of $\Gamma$ on $\Omega^+$ by $f$, we obtain a subgroup $\Gamma_f:=f^{-1}\Gamma f$ acting by biholomorphisms on $\D$. Hence  the quotient of $\D$ by $\Gamma_f$ is homeomorphic to $\Sigma$ and is endowed with a hyperbolic metric induced by the Poincar\'e metric on $\D$, which is nothing but the first Bers' parameter $h^+$ on $\Sigma\cong \D/\Gamma_f$. (The map $f$ is unique up to pre-composition with a biholomorphism of $\D$, therefore the conjugacy class of $\Gamma_f$, and consequently its associated point in $\T$, is independent of the choice of the map $f$.) 

Moreover,   $S(f)$ induces  a holomorphic quadratic differential in the quotient $\D/\Gamma_f\cong\Sigma$. 
Indeed, for any $\gamma\in \Gamma_f$, set $\zeta:=f\circ\gamma\circ f^{-1}\in\Gamma$. Using \eqref{eq:schw} twice we get
$$S(f)=S(\zeta\circ f)=S(f\circ\gamma)=\gamma^*S(f)$$
because $\gamma$ acts on $\D$ as the restriction of a M\"obius transformation of $\C\mathrm P^1$. Hence $S(f)$ is $\Gamma_f$-invariant and passes to the quotient by $\Gamma_f$. We can therefore construct a well-defined map
\begin{equation}\label{eq:schwarzian map}
\mathcal S:\QF(\Sigma)\to\Q(\Sigma)~.
\end{equation}
Here $\Q(\Sigma)$ denotes the bundle of holomorphic quadratic differentials over $\T$, whose fiber over a point $(\Sigma,[h])$ coincides with the vector space $H^0((\Sigma,h),K^2)$, where $K$ denotes the canonical divisor of $(\Sigma,[h])$. Consequently, the space $\Q(\Sigma)$ is a complex manifold of dimension $3g-3$, where $g$ denotes the genus of $\Sigma$. In fact, the map $\mathcal S$ turns out to be injective (see also the discussion below on the construction of its inverse) and, being $\QF(\Sigma)$ and $\Q(\Sigma)$ manifolds of the same real dimension, the invariance of domain theorem implies that its image is an open subset of $\Q(\Sigma)$.

\subsubsection{Constructing the inverse} We will often use the the inverse map of $\mathcal S$, defined on the image of  $\QF(\Sigma)$. Hence it will be useful to quickly discuss its explicit construction. In general, given a holomorphic quadratic differential $q$ on a connected open set $\Omega\subset \C$, there exists a locally injective holomorphic map $f_q:\Omega\to \C$ such that $S(f_q)=q$, see \cite{zbMATH03053379} and \cite[Proposition 6.3.7]{zbMATH05042912}. By the fundamental properties discussed above, $f_q$ is unique up to post-composition with a M\"obius transformation. One can also see that $f_q$, suitably normalized, depends smoothly on $q$. See \cite[Section 3.2]{dumascp} for more details.

To apply this in our setting, we consider a hyperbolic metric $h$ on $\Sigma$ and $\phi\in H^0((\Sigma,h),K^2)$, and realize $(\Sigma,h)$ as the quotient of the Poincar\'e disc $\D$ by a discrete group $\Gamma_h$ of biholomorphisms. We can then lift $\phi$ to a $\Gamma_h$-invariant holomorphic quadratic differential $\tilde \phi$ on $\D$, and find a locally injective holomorphic map $f_{\tilde\phi}:\D\to\C$ whose Schwarzian derivative is equal to $\tilde \phi$. Since $\tilde\phi$ is invariant under the action of $\Gamma_h$, we have
$$S(f\circ \gamma)=\gamma^*S(f)=\gamma^*\tilde\phi=\tilde\phi=S(f)$$ for every $\gamma\in \Gamma_h$. We deduce that for any $\gamma \in \Gamma_h$ there exists a M\"obius tranformation $\zeta=\zeta(\gamma)$ such that $f\circ\gamma=\zeta(\gamma)\circ f$, providing us with a representation $\zeta : \Gamma_h \rightarrow \PSL(2,\C)$. This construction is exactly the inverse of the map $\mathcal S$, in the sense that if $f:\D\to\Omega^+$ is the biholomorphic map associated to a quasi-Fuchsian manifold $\Hyp^3/\Gamma$ as in Section \ref{subsec:schwarzian}, and $h$ and $\phi$ are the induced hyperbolic metric and holomorphic quadratic differential on $\Sigma$, then $f_{\tilde\phi}=f$ and the image of the representation $\zeta$ coincides with the quasi-Fuchsian group $\Gamma$.

\begin{remark} For the reader familiar with complex projective structures on surfaces, $\mathcal S$ is simply the map that associates to a quasi-Fuchsian manifold $M$ the Schwarzian parameterization of the complex projective structure that is naturally defined on $\partial_\infty^+M$, using the fact that the quasi-Fuchsian group $\Gamma$ acts by M\"obius transformations on $\Omega^+$. By the construction described above, we are considering here the Schwarzian parameterization with respect to the Fuchsian section $\T\to\mathcal{CP}(\Sigma)$, where $\mathcal{CP}(\Sigma)$ denotes the deformation space of complex projective structures. Indeed, $\mathcal S$ can be extended to a map from $\mathcal{CP}(\Sigma)$ to $\Q(\Sigma)$ which is a homeomorphism; the construction above provides the inverse $\Q(\Sigma)\to\mathcal{CP}(\Sigma)$. We remark  that the map $f_{\tilde\phi}$ is not univalent in general; it is univalent when it arises from a pair $(h,\phi)$ in the image of $\QF(\Sigma)$. However, we will not need to deal with general complex projective structures in this paper, hence we will not adopt this language.
\end{remark}

\subsection{Epstein surfaces}
In this subsection we describe a construction due to Epstein in \cite{epstein}, which naturally associates to certain conformal metrics on a domain of $\mathbb{C}P^{1}\cong \partial_{\infty}\mathbb{H}^{3}$ an immersion into $\mathbb{H}^{3}$, that we will call the Epstein surface. 

\subsubsection{The Epstein map} Given any point $p\in \mathbb{H}^{3}$, we define a map  $G_p:T^1_p\Hyp^3\to\mathbb{C}P^{1}$, by sending $(x,v)$ to the endpoint at infinity of the unique geodesic of $\Hyp^3$ starting at $x$ with tangent vector $v$. Then we define the \emph{visual metric} $V_p$ as the metric obtained by pushforward via $G_p$ of the canonical spherical metric of $T^1_p\Hyp^3$. One can easily check that the metric $V_p$ is conformal, namely compatible with the Riemann surface structure of $\mathbb{C}P^{1}$. Indeed, if $o$ is the origin in the unit ball model, then $V_o$ is just the usual spherical metric on the unit sphere. For the general case, if $M$ is an isometry of $\Hyp^3$ sending $o$ to $p$, then $V_p=M_*V_o$ and is therefore in the same conformal class, since $M$ extends to a biholomorphism of $\mathbb{C}P^{1}$.

The fundamental result is the following:

\begin{prop}[\cite{epstein,Dumas2017}]\label{prop:eps map}
Let $\Omega$ be a connected open domain in $\mathbb{C}P^{1}$ and let $\varphi:\Omega\to \mathbb{C}P^{1}$ be a locally injective holomorphic map. If $\sigma$ is a $C^1$ conformal metric on $\Omega$, then there exists a unique continuous map $\Eps_{(\varphi,\sigma)}:\Omega\to\Hyp^3$ such that
$$(\varphi^*V_{\Eps_{(\varphi,\sigma)}(z)})(z)=\sigma(z)$$
for all $z\in \Omega$.  Moreover, if $\sigma$ is $C^k$, then $\Eps_{(\varphi,\sigma)}$ is $C^{k-1}$.
\end{prop}



We remark that $\Eps_{(\varphi,\sigma)}$ is in general not an immersion. As an example, if $\sigma$ is the standard spherical metric on the unit sphere, then the associated Epstein map is constantly equal to the origin $o$ in the unit ball model.

In \cite[Section 3]{Dumas2017} Dumas introduced an explicit formula for $\Eps_{(\varphi,\sigma)}$ in the upper half-space model of $\Hyp^3$, which will be useful for our purposes. Let $p$ be the point in the geodesic joining $0$ and $\infty$ in the upper half-space model such that the visual metric $V_p$ at $0$ equals $|dz|^2$. Concretely, $p=(0,0,2)$. If we write the conformal metric as $\sigma=e^{2\eta}|dz|^2$, and  to simplify the notation we let $\Omega$ be a connected open subset of $\C$ so as to take $\varphi=\mathrm{id}$, then the expression for $\Eps_{(\mathrm{id},\sigma)}:D\to\Hyp^3$ is the following:

	\begin{equation}\label{eq:SLframe}
	\Eps_{(\mathrm{id},\sigma)}(z)=\begin{pmatrix}
			1 & z \\
			0& 1 \\
		\end{pmatrix}\begin{pmatrix}
			1 & 0 \\
			\eta_{z}& 1 \\
		\end{pmatrix}\begin{pmatrix}
			e^{-\frac{\eta}{2}} & 0 \\
			0& e^{\frac{\eta}{2}} \\
	\end{pmatrix}\cdot p \end{equation}

\subsubsection{Schwarzian tensors}
The last fundamental preliminary step that we will need in our paper is an expression for the mean curvature of Epstein maps. For this purpose, we first need to introduce the notion of Schwarzian tensor, due to Osgood and Stowe \cite{zbMATH00078553}. Given two conformal metrics $\sigma_{1}=e^{2\eta_{1}}|dz|^{2}$ and $\sigma_{2}=e^{2\eta_{2}}|dz|^{2}$ on a domain $\Omega \subset \mathbb{C}P^{1}$,  the \emph{Schwarzian tensor} of $\sigma_{1}$ with respect to $\sigma_{2}$  is the  quadratic differential (which is not necessarily holomorphic, in general) defined as  
\begin{equation}\label{eq:schw tensor}
	B(\sigma_{1},\sigma_{2})= ((\eta_{2})_{zz}-{(\eta_{2})_{z}}^{2}- (\eta_{1})_{zz}+{(\eta_{1})_{z}}^{2})dz^{2}
\end{equation}
This definition generalizes the classical Schwarzian derivative, in the sense that, if $f:\Omega\to \C$ is a locally injective holomorphic function, then
\begin{equation}\label{eq:schw tensor and schwarzian}
	S(f) =2B(|dz|^{2},f^{*}|dz|^{2})~.
\end{equation}  
Clearly $B(\sigma_{2},\sigma_{1})=-B(\sigma_{1},\sigma_{2})$. Similarly to the Schwarzian derivative, the Schwarzian tensor has a number of naturality  properties. For any metrics $\sigma_{1},\sigma_{2},\sigma_{3}$ on $\Omega\subset \mathbb{C}P^{1}$, 
\begin{itemize}
	\item Given a locally injective holomorphic map $f$, 
	\begin{equation} \label{eq:swt1} f^{*}B(\sigma_{1},\sigma_{2})=B(f^{*}\sigma_{1},f^{*}\sigma_{2})~.\end{equation}
	\item The cocycle property holds: 
	\begin{equation} \label{eq:swt2}B(\sigma_{1},\sigma_{3})=B(\sigma_{1},\sigma_{2})+B(\sigma_{2},\sigma_{3})~.\end{equation}
\end{itemize}
In particular, \eqref{eq:swt1} implies that if $\sigma_1$ and $\sigma_2$ are invariant by an automorphism of $\Omega$, then so is the quadratic differential $B(\sigma_{1},\sigma_{2})$. If a group $\Gamma$ acts on $\Omega$ by biholomorphisms with $\Omega/\Gamma\cong\Sigma$, thus inducing in quotient surface $\Sigma$ a Riemann surface structure, and $\sigma_1,\sigma_2$ are $\Gamma$-invariant conformal metrics, then $B(\sigma_{1},\sigma_{2})$ induces a well-defined quadratic differential in the quotient.

\subsubsection{M{\"o}bius flat metrics}\label{subsec:mob flat}
A conformal metric $\sigma$ is said to be \emph{M{\"o}bius flat} if $B(\sigma,|dz|^{2})=0$. From \eqref{eq:schw tensor and schwarzian}, for example, when $f$ is itself a M{\"o}bius transformation, then the pull-back metric $f^{*}|dz|^{2}$ is always M{\"o}bius flat. This is not the only case. Indeed, one can show that 
$B(\sigma,|dz|^2)=0$ if and only if $\sigma$ is the pull-back by a M\"obius transformation of one of the following metrics:
\begin{itemize}
\item the flat metric $|dz|^2$ on $\C$,
\item a positive multiple of the Poincar\'e metric on $\D$,
\item a positive multiple of the spherical metric on $\CP^{1}$.
\end{itemize}

Now, given a metric $\sigma$, we will denote by 
$$B(\sigma)=B(g_{\CP^1},\sigma)$$ the Schwarzian tensor of $\sigma$ with respect to a M{\"o}bius flat metric $g_{\CP^{1}}$. By the definition of  M{\"o}bius flat and the cocycle property \eqref{eq:swt2}, $B(\sigma)$ is independent of the chosen  M{\"o}bius flat metric $g_{\CP^{1}}$. Hence if $f$ is a M\"obius transformation, then 
\begin{equation}\label{eq invariance B}
B(f^*\sigma)=f^*B(\sigma)
\end{equation}
 by \eqref{eq:swt1}. As another consequence of the independence of the definition of $B(\sigma)$ from the choice of $g_{\CP^{1}}$, together with the definition \eqref{eq:schw tensor} applied to $B(e^{2t}\sigma)=B(|dz|^2,e^{2t}\sigma)$, we have that if $e^{2t}$ is any positive constant then 
\begin{equation}\label{eq:scale invariance}
B(e^{2t}\sigma)=B(\sigma)~.
\end{equation}

Finally, given a quadratic differential $\phi=\lambda(z)dz^2$ and a conformal metric $\sigma=e^{2\eta}|dz|^2$, we define the norm of $\phi$ with respect to $\sigma$ as:
$$\|\phi\|_\sigma(z):=e^{-2\eta(z)}|\lambda(z)|~.$$
Since both $|\phi|$ and $\sigma$ follow the same transformation rule under a biholomorphic change of coordinates, $\|\phi\|_\sigma$ is as well-defined function, meaning that if $f$ is a locally injective holomorphic function, then 
\begin{equation}\label{eq invariance norm}
\|f^*\phi\|_{f^*\sigma}=\|\phi\|_\sigma\circ f~.
\end{equation}
 In particular, if $\sigma=e^{2u}h_0$ is a conformal metric on  $(\Sigma,h)$ and $\phi$ is a quadratic differential on  $(\Sigma,h)$, then $\|\phi\|_\sigma$ is a function on $\Sigma$. From \eqref{eq:scale invariance}, we also obtain:
\begin{equation}\label{eq:scale invariance2}
\|\phi\|_{e^{2t}\sigma}=e^{-2t}\|\phi\|_\sigma~,
\end{equation}
for any constant $t\in\R$.

\subsubsection{Mean curvature}\label{subsec:mean}
We are now ready to provide the formula for the mean curvature of  Epstein maps. Let $\sigma$ be a $C^2$ conformal metric on an open set $\Omega$. To simplify the notation, we first suppose $\varphi=\mathrm{id}$. Assume moreover that $\Eps_{(\mathrm{id},\sigma)}$ is an immersion. In this case, it turns out that $\Eps_{(\mathrm{id},\sigma)}$ at $z$ is tangent to the unique horosphere through $\Eps_{(\mathrm{id},\sigma)}$ with point at infinity $z$. Then, the mean curvature of $\Eps_{(\mathrm{id},\sigma)}$ equals the function
	
	\begin{equation}\label{meancurveq}
		\mathcal H(\Eps_{(\mathrm{id},\sigma)}) = \frac{K(\sigma)^{2}-1-16\|B(\sigma)\|_\sigma^2}{(K(\sigma)-1)^{2}-16\|B(\sigma)\|_\sigma^2}~,
	\end{equation}
	where $K(\sigma)$ denotes the curvature of $\sigma$. See \cite[Equations 3.2, 3.3]{dumascp} and \cite[Lemma 3.4]{quinn}.
Here the mean curvature is defined as one half the trace of the second fundamental form with respect to the first fundamental form. It is computed with respect to the unit normal vector pointing towards $\Omega$. We will then apply the formula \eqref{meancurveq} when the Epstein map induces an embedded surface in $\Hyp^3/\Gamma$ for $\Gamma$ a quasi-Fuchsian group, and for $\Omega=\Omega^+$. Hence the convention of the mean curvature here is consistent with Remark \ref{rmk:convention}.

To write the general formula for $\Eps_{(\varphi,\sigma)}$, since the computation is local, we may restrict to an open subset $\Omega$ on which $\varphi$ is a biholomorphism onto its image. Let $\sigma$ be a metric on $\Omega$ and $\hat\sigma$ be such that $\varphi^*\hat\sigma=\sigma$. Then we observe that $K(\hat\sigma)\circ \varphi=K(\sigma)$, whereas by \eqref{eq:schw tensor and schwarzian}, \eqref{eq:swt1} and \eqref{eq:swt2},
$$\varphi^*B(\hat\sigma)=B(\varphi^*g_{\CP^1},\sigma)=B(\varphi^*g_{\CP^1},g_{\CP^1})+B(g_{\CP^1},\sigma)=B(\sigma)-\frac{1}{2}S(\varphi)~.$$
Hence we can deduce the expression:
\begin{equation}\label{meancurveq2}
		\mathcal H(\Eps_{(\varphi,\sigma)}) = \frac{K(\sigma)^{2}-1-16\|B(\sigma)-S(\varphi)/2\|^2_\sigma}{(K(\sigma)-1)^{2}-16\|B(\sigma)-S(\varphi)/2\|^2_\sigma}
	\end{equation}

\subsubsection{A technical point}
The rough idea to prove the existence of CMC surfaces using the implicit function theorem is the following. Consider quasi-Fuchsian manifolds $\Hyp^3/\Gamma$, where $\Omega^\pm$ are the connected components of the complement of the limit set $\Lambda_\Gamma$. We would like to write the solutions of the CMC condition $\mathcal H=c$, for $c\in (-1,1)$, as the level sets of a function $G$ which depends on the hyperbolic metric $h$ on $\Sigma$ in the conformal class of $\Omega^+/\Gamma$ (that is, it represents the first Bers parameter $h^+$ of $M$), on a holomorphic quadratic differential $\phi$ on $(\Sigma,h)$ which is (the quotient of) the Schwarzian derivative of the conformal isomorphism between $\D$ and $\Omega_+$, and finally on the conformal factor of a metric of the form $e^{2u}h$ on $\Sigma$. This last function $u$ is  an element of the infinite-dimensional functional space $C^\infty(\Sigma,\R)$. A priori the pair $(h,\phi)$ varies in an infinite-dimensional space as well, since $h$ varies in the space $\mathcal M_{-1}(\Sigma)$ of hyperbolic metrics. Although this is not really necessary, it will be convenient to use the action of $\mathrm{Diff}_0(\Sigma)$  to reduce ourselves to representatives of pairs $(h,\phi)$, now varying in the finite-dimensional space $\Q(\Sigma)$. The following lemma will serve to formalize this approach.

\begin{lemma}\label{lem:section}
Let $\pi:\mathcal M_{-1}(\Sigma)\to\T$ be the quotient map by the action of $\mathrm{Diff}_0(\Sigma)$ on $\mathcal M_{-1}(\Sigma)$. There exists a smooth section $s:\T\to \mathcal M_{-1}(\Sigma)$ of $\pi$.
\end{lemma}

We remark that the section $s$ that we are looking for is not ``canonical" in any manner. There are actually several ways to achieve this; we will sketch one relying on the theory of harmonic maps of hyperbolic surfaces, see \cite{Wolf1989}.

\begin{proof}[Sketch of proof of Lemma \ref{lem:section}]
Fix a hyperbolic metric $h_0$ on $\Sigma$, and consider the vector space $H^0((\Sigma,h_0),K^2)$ of holomorphic quadratic differentials on $(\Sigma,h_0)$. Then for every $q\in H^0((\Sigma,h_0),K^2)$ there exists a unique hyperbolic metric $h_q$ such that $\mathrm{id}:(\Sigma,h_0)\to (\Sigma,h_q)$ is harmonic, with $h_q$ depending smoothly on $q$. The correspondence $q\mapsto h_q$ therefore gives a map $H^0((\Sigma,h_0),K^2) \rightarrow \mathcal M_{-1}(\Sigma)$ that, when post-composed with $\pi$, provides a homeomorphism from $H^0((\Sigma,h_0),K^2)$ to $\T$. This proves the existence of the desired section.
\end{proof}

\begin{remark}\label{rmk:wolf}
Wolf's approach via harmonic maps actually led to the construction of a global parameterization of $\T$ by means of the space $H^0((\Sigma,h_0),K^2)$, once the metric $h_0$ is fixed. This allows us to identify the space $\Q(\Sigma)$ with a very concrete finite-dimensional manifold of real dimension $12g-12$, namely the total space of the smooth vector bundle $\mathcal E$ over $H^0((\Sigma,h_0),K^2)$ whose fiber over a quadratic differential $q$ is equal to $H^0((\Sigma,h_q),K^2)$, the space of holomorphic quadratic differentials of the hyperbolic surface $(\Sigma,h_q)$. In rest of our exposition we will identify with abuse any pair $(h,\phi)$ with its corresponding point in the total space of $\mathcal E$. (Notice that the identification with $\Q(\Sigma)$ heavily depends on the choice of the section $s$ from Lemma \ref{lem:section}.)
\end{remark}

\section{Existence of CMC surfaces}

The purpose of this section is to prove two existence results for CMC surfaces, morally one (Theorem \ref{thm:foliation_ends}) ``in the ends" and the other (Theorem \ref{thm:existence_compact}) ``in the compact part". Then in Theorem \ref{thm:existence} we combine them to obtain the existence of CMC surfaces for $h\in(-1,1)$ for quasi-Fuchsian manifolds close to the Fuchsian locus, which is for the moment  weaker than our main result, Theorem \ref{thm:foliation}.

\subsection{Existence in the ends}\label{sec:foliation_ends}

It has been proved in \cite{MP} that the ends of every quasi-Fuchsian manifold are monotonically foliated by CMC surfaces; another proof has been provided recently in \cite{quinn}. Here we will need an improved statement, so as to have a local (in $\QF(\Sigma)$) uniform control on the value of the mean curvature along the leaves of the foliation. 

\begin{theorem}\label{thm:foliation_ends}
	Let $\Sigma$ be a closed oriented surface of genus $\geq 2$ and $m\in\QF(\Sigma)$. Then there exists a neighbourhood $U_0$ of $m$ in $\QF(\Sigma)$ and a constant $\epsilon=\epsilon(m,U_0)$ such that the ends of every quasi-Fuchsian manifold in $U_0$ are smoothly monotonically foliated by CMC surfaces whose mean curvature ranges in $(-1,-1+\epsilon)$ and in $(1-\epsilon,1)$.
\end{theorem}

We say that the ends of $M\cong\Sigma\times\R$ are the connected components of the complement of a compact submanifold with boundary in $M$ homeomorphic to $\Sigma\times I$ for $I$ a closed interval.

\subsubsection{Outline of the CMC existence for a fixed manifold}\label{subsec:outline1}
We now quickly review, using our notation and set-up, the proof given in \cite{quinn} and later we will explain how it adapts in order to prove Theorem \ref{thm:foliation_ends}. Roughly speaking, the proof of \cite{quinn} is  an application of the implicit function theorem to the equation of constant mean curvature from the mean curvature formula \eqref{meancurveq}, with respect to a conformal metric at infinity.

More precisely, the idea of Quinn's  proof is to consider Epstein maps defined on $\Omega_+$, with $\varphi=\mathrm{id}$, associated to  a conformal metric of the form $\sigma(u)=e^{2u}h_0$ for $h_0$ the conformal complete hyperbolic metric, and to study the following equation  in $u$:
$$\mathcal H(\Eps_{(\mathrm{id},\sigma(u))})=H$$
 for $H\in (-1,1)$ close to $\pm 1$. From \eqref{meancurveq}, this gives  the equation: 
\begin{equation}\label{eq:quinn}
	\mathcal H(\Eps_{(\mathrm{id},\sigma(u))}) = \frac{K(\sigma(u))^{2}-1-16\|B(\sigma(u))\|^{2}_{\sigma(u)}}{(K(\sigma(u))-1)^{2}-16\|B(\sigma(u))\|^{2}_{\sigma(u)}}= H
\end{equation}

\begin{remark}\label{rmk:quinn1}
If we choose a  metric $\sigma$ invariant under the quasi-Fuchsian group $\Gamma$ acting on $\Omega_+$ by biholomorphisms, then $ \|B(\sigma(u))\|^{2}_{\sigma(u)}$ is a well-defined invariant function on the quotient $\Sigma$, by \eqref{eq invariance B} and \eqref{eq invariance norm}. This shows that the equation \eqref{eq:quinn} can be really thought as an equation for a function $u$ on the quotient surface $\Sigma$, where $\sigma(u)=e^{2u}h_0$ is a metric on $\Sigma$.
\end{remark}

\begin{remark}\label{rmk:quinn2}
In the situation of Remark \ref{rmk:quinn1}, the uniqueness property of the Epstein map as in Proposition \ref{prop:eps map} implies that the Epstein surface is invariant under the quasi-Fuchsian group $\Gamma$. More precisely, for any $\gamma\in \Gamma$, we have
\begin{equation}\label{eq:invariance}
\Eps_{(\mathrm{id},\sigma(u))}\circ\gamma=\gamma\circ \Eps_{(\mathrm{id},\sigma(u))}~.
\end{equation}
Therefore $\Eps_{(\mathrm{id},\sigma(u))}$ induces a map from $\Omega^+/\Gamma$ to the quasi-Fuchsian manifold $\Hyp^3/\Gamma$.
\end{remark}

Now the trick consists in performing a renormalization to Equation \eqref{eq:quinn}, so as to obtain an equivalent equation, for which we can find an explicit solution for $H=-1$. This consists in the change of variables from $(H,u)$ to $(H,v)$, where
\begin{equation}\label{eq:change variables}
v:=u+\frac{1}{2}\log\left(\frac{1+H}{1-H}\right)~.
\end{equation}
Let us now set $\tau(v)=e^{2v}h_0$, so that we have the identity:
\begin{equation}\label{eq:change variables2}
\tau(v)=\frac{1+H}{1-H}\sigma(u)~.
\end{equation}
A direct computation from \eqref{eq:quinn} and \eqref{eq:change variables2} (and using also \eqref{eq:scale invariance2}) shows that $u$ solves \eqref{eq:quinn} for $H\in (-1,1)$ if and only if $v$ solves the equation:

\begin{equation}\label{eq:quinn2}
	G(H,v):= 1-H -2HK(\tau(v) ) + (-1-H)\left(K(\tau(v))^{2}- 16\|B(\tau(v))\|^{2}_{\tau(v)}\right)=0~.
\end{equation}

The big advantage is that now the choice $v_0\equiv 0$ satisfies $G(-1,v_0)=0$, since $\tau(v_0)=h_0$  and $K(h_0)=-1$.
Hence we are in the right setting to apply the implicit function theorem near this solution $(-1,v_0)$ of the equation $G=0$ (see e.g. \cite[\S I.5]{lang}). One must show that the derivative of $G$ with respect of $u$ is an invertible operator between suitable function spaces (see details below), and achieves a family of solutions $\mathsf v=\mathsf v(H)$ of \eqref{eq:quinn2} depending smoothly on $H$, for $H\in [-1,-1+\epsilon)$. This will provide CMC surfaces with mean curvature $H$ close to $-1$ via the Epstein maps $\Eps_{(\mathrm{id},\sigma(\mathsf u(H)))}$, where
$$\mathsf u(H)=\mathsf v(H)-\frac{1}{2}\log\left(\frac{1+H}{1-H}\right)~.$$

\subsubsection{Adaptation for Theorem \ref{thm:foliation_ends}}\label{subsec:adapt}

We will now describe the extension of this strategy in our setting. The difference is that we need to allow the quasi-Fuchsian manifold to vary as well, represented by a variation of a pair $(h,\phi)$, and thus of the holomorphic map $f=f_{\tilde\phi}$ which gives a biholomorphism between $\D$ and the domain $\Omega^+$. Let us explain this in detail.

To make explicit the dependence on the hyperbolic metric $h$, we now denote $\sigma_h(u):=e^{2u}h$. We need to replace Equation \eqref{eq:quinn} by the condition that the mean curvature of the Epstein map $\Eps_{(f_{\tilde\phi},\sigma_h(u))}$ equals $H$. From Equation \eqref{meancurveq2}, we see that such identity reads:

\begin{equation}\label{eq:us}
	 \frac{K(\sigma_h(u))^{2}-1-16\|B(\sigma_h(u))-\phi/2\|^{2}_{\sigma_h(u)}}{(K(\sigma_h(u))-1)^{2}-16\|B(\sigma_h(u))-\phi/2\|^{2}_{\sigma_h(u)}}= H
\end{equation}
where we have used that the holomorphic quadratic differential induced in the quotient by $S(f_{\tilde\phi})$ equals $\phi$ by construction. This is again an equation on the closed surface $\Sigma$, and the same change of variables as in \eqref{eq:change variables} leads to the equation:

\begin{equation}\label{eq:us2}
	G(H,h,\phi,v):= 1-H -2HK(\tau_h(v) ) + (-1-H)\left(K(\tau_h(v))^{2}- 16\|B(\tau_h(v))-\phi/2\|^{2}_{\tau_h(v)}\right)=0~,
\end{equation}
where now $\tau_h(v)=e^{2v}h$. 

Now, fix a hyperbolic metric $h_0$ on $\Sigma$ and a holomorphic quadratic differential $\phi_0$ on $(\Sigma,h)$. Similarly to Section \ref{subsec:outline1}, a solution to Equation \eqref{eq:us2} is given by $(-1,h_0,\phi_0,v_0)$ where $v_0$ denotes the constant null function, since $\tau_{h_0}(v_0)=h_0$ has curvature $-1$. 
To apply the implicit function theorem, let us describe carefully the domain of definition of $G$. Recall from Remark \ref{rmk:wolf} that the choice of a section as in Lemma \ref{lem:section} provides us with a diffeomorphism between $\Q(\Sigma)$ and  $\R^{12g-12}$. We consider thus the open subset $W$ of $\R^{12g-12}$ that corresponds to the image of $\QF(\Sigma)$ under the map $\mathcal S$ introduced in  \eqref{eq:schwarzian map}. By a small abuse of notation, we will denote the elements of $W$ as a pair $(h,\phi)$, where $h$ is a hyperbolic metric and $\phi$ a holomorphic quadratic differential on $(\Sigma,h)$. Then we consider $G$ as a map
$$G:\R\times W\times W^{2,s}(\Sigma, h_0)\to W^{2,s - 2}(\Sigma, h_0)$$
for $s\geq 2$, where $W^{2,s}(\Sigma, h_0)$ denotes the Sobolev space of real-valued functions on $\Sigma$ that admit $L^2$-integrable weak derivatives of order $\leq s$ (with respect to the standard Riemannian measure of $h_0$), and $W^{2,0}(\Sigma,h_0) : = L^2(\Sigma, h_0)$. By direct inspection, $G$ depends smoothly on all variables. 
We now need to show that the derivative $d_{v}G_{(-1,h_0,\phi_0,v_0)}$ is a bounded invertible operator, for any $s\geq 2$. A simple computation gives: 

\begin{equation}\label{eq:differential}
\begin{split}
	d_{v}G_{(-1,h_0,\phi_0,v_0)}(\dot v) &= 2 \left.\frac{d}{dv}\right|_{v=v_0} (K(e^{2v}h)) \\
	& = 2 \left.\frac{d}{dv}\right|_{v=v_0}(e^{-2v}(-\Delta_{h_0} v+K(h_0))  \\
	&= 2(2\dot v-\Delta_{h_0}\dot v)
\end{split}
\end{equation} 
It is well-known that such an operator is a continuous linear isomorphism; we provide here a sketch of proof for convenience of the reader.

\begin{lemma}\label{lem filippo}
Let $f$ be a smooth and strictly positive function, and let $h$ be any Riemannian metric on a compact surface $\Sigma$. Then the operator $u \mapsto f u - \Delta_h u$ is a positive definite and continuous linear isomorphism from $W^{2,s}(\Sigma, h)$ to $W^{2,s-2}(\Sigma, h)$ for any $s\geq 2$. In particular, for any smooth function $\lambda$ on $\Sigma$, there exists a unique smooth function $u$ satisfying $\Delta_h u - f u = \lambda$.
\end{lemma}
\begin{proof}
    Let $T$ denote the continuous linear operator
    \[
    T : = f \, \mathit{id} - \Delta_h : W^{2,s}(\Sigma, h) \to W^{2, s - 2}(\Sigma, h) ,
    \]
    for some $s \geq 2$. A simple integration by parts shows that $T$ is a positive definite symmetric operator with respect to the $L^2$-scalar product: indeed, for any $v, w \in W^{2,s}(\Sigma, h)$, we have
    \[
    \langle v, T w \rangle_{L^2} = \int_{\Sigma} v \, T w \, \mathrm{d}{a}_h = \int_{\Sigma} (f v w + h(\nabla v, \nabla w) ) \, \mathrm{d}{a}_h ,
    \]
    where $\nabla v$ denotes the (weak) gradient of $v$ with respect to the metric $h$, and $\mathrm{d}{a}_h$ is the standard Riemannian volume form. Since $f$ is a strictly positive function, $T$ satisfies $\langle v , T v \rangle_{L^2} \geq 0$ for any $v \in W^{2,s}(\Sigma, h)$, with equality if and only if $v = 0$. To prove that $T$ is surjective, let $\lambda \in W^{2,s-2}(\Sigma, h)$, and define the linear functional
    \[
    \varphi(v) : = \int_{\Sigma} v\lambda \, \mathrm{d}{a}_h .
    \]
    Notice that $\varphi$ is continuous with respect to the $L^2$-norm, and hence with respect to the Sobolev norm $\| \cdot \|_{W^{2,s}}$ for any $s \geq 0$. We now introduce the following bilinear symmetric form on $W^{2,1}(\Sigma,h)$:
    \[
    a(v,w) : = \int_{\Sigma} (f vw + h(\nabla v, \nabla w)) \, \mathrm{d}{a}_h ,
    \]
    If $C \geq 1$ is some positive constant satisfying $C^{-1} \leq f \leq C$, then we have
    \[
    C^{-1} \| v \|^2_{W^{2,1}} \leq a(v,v) \leq C \| v \|^2_{W^{2,1}} .
    \]
    for any $v \in W^{2,1}(\Sigma,h)$. Therefore the bilinear form $a$ is equivalent to the standard Hilbert scalar product of the Sobolev space $W^{2,1}(\Sigma,h)$, and therefore $\varphi$ is continuous with respect to $a$ as well. By Riesz representation theorem, we conclude that there exists a unique $u \in W^{2,1}(\Sigma,h)$ satisfying $a(u,v) = \varphi(v)$ for any $v \in W^{2,1}(\Sigma,h)$. This proves the existence of a weak solution $u \in W^{2,1}(\Sigma,h)$ of the equation $f u - \Delta_h u = \lambda$. 
    
    A more delicate analysis is then required to show that the regularity of $\lambda \in W^{2,s - 2}(\Sigma, h)$  is sufficient to "promote" $u$ to a genuine element in $W^{2,s}(\Sigma, h)$ satisfying $T u = \lambda$. This is the part of the argument where elliptic regularity theory is required, leading to controls of the form
    \[
    \| u \|_{W^{2,s}} \leq M ( \| u \|_{L^2} + \| \lambda \|_{W^{2,s - 2}} ) ,
    \]
    with the multiplicative constant $M > 0$ that depends only on $s \geq 2$, the function $f$, and the compact Riemannian surface $(\Sigma, h)$. We refer to \cite[\S 10.3.2]{nicolaescu} (see in particular \cite[Theorem 10.3.12]{nicolaescu}) for a detailed exposition of elliptic regularity results on smooth manifolds.
\end{proof}

We have thus shown that $d_{v}G:W^{2,s}(\Sigma,h_0)\to W^{2,s-2}(\Sigma,h_0)$ is a linear isomorphism at the point $(-1,h,\phi,v_0)$. We can now apply the implicit function theorem for Banach spaces, and deduce that there exist $\epsilon>0$, a neighbourhood $U_0$ of $(h_0,\phi_0)$ and a function 
$$\mathsf v:[-1,1+\epsilon)\times U_0\to W^{2,s}(\Sigma,h_0)$$
 such that all solutions of $G=0$ in a neighbourhood of $(-1,h_0,\phi_0,v_0)$ are of the form $G(H,h,\phi,\mathsf v(H,h,\phi))=0$. Exactly as in \cite{quinn}, one can then apply elliptic regularity to show that the functions $\mathsf v(H,h,\phi)$ are smooth and depend smoothly on $(H,h,\phi)$ (see e.g. \cite[Lemma 17.16]{trudinger}).

Using \eqref{eq:change variables}, we then define the function $\mathsf u:[-1,1+\epsilon)\times U_0\to W^{2,s}(\Sigma,h)$ by 
\begin{equation}\label{eq solutions u}
\mathsf u(H,h,\phi):=\mathsf v(H,h,\phi)-\frac{1}{2}\log\left(\frac{1+H}{1-H}\right)~.
\end{equation}
By construction, as $H$ varies in $[-1,-1+\epsilon)$, the Epstein maps 
$$\Eps_{(f_{\tilde\phi},e^{2\mathsf u(H,h,\phi)})}:\D\to\Hyp^3$$ then induce (smooth) CMC immersions of mean curvature $H$. We will  see in Section \ref{sec:fol ends} below that, up to choosing smaller $\epsilon$ and $U_0$, these maps are immersions. Moreover, as observed in Remark \ref{rmk:quinn2}, they induce CMC immersions in the quasi-Fuchsian manifold whose image via the map $\mathcal S$ is the point $(h,\phi)$.

Of course the same argument can be applied to the other end, namely for the component $\Omega^-$ of the domain of discontinuity, and for $H$ close to $1$. This concludes the existence part in Theorem \ref{thm:foliation_ends}.

\subsection{Foliations of the ends}\label{sec:fol ends}

We now discuss the foliation part of Theorem \ref{thm:foliation_ends}.
For this purpose, let us first outline the proof given in \cite{quinn}, to show that the ends of a given quasi-Fuchsian manifold $M$ are foliated by CMC surfaces; we will then adapt this proof in order to complete the proof of Theorem \ref{thm:foliation_ends}.

\subsubsection{Outline of the foliation statement for a fixed manifold}

In our notation from the previous section, Quinn's idea is to consider, for $h_0$ and $\phi_0$ fixed, the map 
$$\widehat\Psi:\Sigma\times [-1,-1+\epsilon)\to M\cup\partial_\infty^+M$$
which is induced in the quotient by the map $\Psi:\Omega^+\times [-1,-1+\epsilon)\to \Hyp^3\cup\Omega^+$:
$$\Psi(z,H)=\begin{cases}
z & \textrm{if }H=-1 \\
\Eps_{(\mathrm{id},e^{2\mathsf u(H)})}(z) & \textrm{if }H>-1
\end{cases}$$
Then one would like to show that $\Psi$ is a local diffeomorphism at every $(z,-1)$, and use a compactness argument to deduce that $\widehat\Psi$ is a diffeomorphism from $\Sigma\times [-1,-1+\epsilon')$ onto its image, up to choosing $\epsilon'<\epsilon$ sufficiently small.

Unfortunately, the differential of the map $\Psi$ written above is not a injective at the points $(z,-1)$. However, this is easily fixed by a reparameterization of the  parameter $H$. Set $t(H)=\sqrt{1+H}$, and write $H(t)=-1+t^2$ for $t>0$. Then we modify the map $\Psi$ above to a new map, that we call again $\Psi:\Omega^+\times [0,\delta)\to \Hyp^3\cup\Omega^+$ with an abuse of notation, for $\delta=\sqrt{1+\epsilon}$. It is defined by:
\begin{equation}\label{eq:map corrected}
\Psi(z,t)=\begin{cases}
z & \textrm{if }t=0 \\
\Eps_{(\mathrm{id},e^{2\mathsf u(H(t))})}(z) & \textrm{if }t>0
\end{cases}
\end{equation}

The map in \eqref{eq:map corrected} is now the expression that we would like to differentiate at points $(z,t=0)$. This is easily done using the following explicit expression for the Epstein map when $\varphi=\mathrm{id}$ and $\sigma=e^{2\eta}|dz|^2$, which is a consequence of the formula \eqref{eq:SLframe}:

\begin{align*}
		\Eps_{(\mathrm{id},\sigma)}(z)= (z,0)+ \frac{2}{e^{2\eta}+4|\eta_{z}|^{2}}(2\eta_{\bar z},e^{\eta})~.
	\end{align*}
We must apply this formula to the metric $\sigma(u)=e^{2\eta}|dz|^2=e^{2u}h_0$, for 
$$u=\mathsf u(H(t))=\mathsf v(H(t))-\frac{1}{2}\log\left(\frac{1+H(t)}{1-H(t)}\right)$$
as in \eqref{eq solutions u}. Writing $v=\mathsf v(H(t))$ and  $\tau(v)=e^{2v}h_0=e^{2\lambda}|dz|^2$, we have 
$$\eta=\lambda-\frac{1}{2}\log\left(\frac{1+H(t)}{1-H(t)}\right)$$
and therefore we obtain the expression:

\begin{equation}\label{eq:long}
\begin{split}
		\Eps_{(\mathrm{id},\sigma(\mathsf u(H(t)))}(z)&= (z,0)+ \frac{2}{e^{2\lambda}+4\frac{1+H(t)}{1-H(t)}|\lambda_{\bar z}|^{2}}\left(2\frac{1+H(t)}{1-H(t)}\lambda_{z},\sqrt{\frac{1+H(t)}{1-H(t)}}e^{\lambda}\right) \\
		&=(z,0)+ \frac{2}{e^{2\lambda}+\frac{4t^2}{2-t^2}|\lambda_{z}|^{2}}\left(\frac{2t^2}{2-t^2}\lambda_{\bar z},\sqrt{\frac{t^2}{2-t^2}}e^{\lambda}\right)~.
	\end{split}
	\end{equation}
From here, one sees that the limit as $t\to 0^+$ (that is, as $H\to -1^+$) of  $\Eps_{(\mathrm{id},\sigma(\mathsf u(H(t))))}(z)$ equals $z$. Moreover, the derivative of $\Eps_{(\mathrm{id},\sigma(\mathsf u(H(t)))}$ with respect to $t$ at $t=0$ equals $(0,\sqrt 2e^{-\varrho})$ where $\varrho$ is the density of the hyperbolic metric on $\Omega^+$ with respect to $|dz|^2$. Therefore we have
(in real coordinates on the upper half-space): 

\begin{equation}\label{eq:invertible matrix 1}
	d\Psi_{(z,0)}=\begin{pmatrix}
		1
		 & 0 & 0 \\
		0 & 1 &
		0  \\
	
		0 & 0 &	\sqrt{2}e^{-\varrho}
		
	\end{pmatrix}
	\end{equation}
which is clearly invertible.

\subsubsection{Adaptation for Theorem \ref{thm:foliation_ends}}\label{subsec:adapt2}

The above construction by Quinn is analogue to the one that we apply here, up to a modification in order to be able to choose $\epsilon'$ uniformly when
the pair $(h,\phi)$ varies in a small neighbourhood of $(h_0,\phi_0)$. For this purpose, we modify the maps above (which we denote with the same symbol by a small abuse of notation) to:

$$\Psi:\D\times [0,\delta)\times U_0\to \left(\Hyp^3\cup\partial_\infty\Hyp^3\right)\times U_0$$
defined by (recall the definition of $u(H,h,\phi)$ in \eqref{eq solutions u}):
\begin{equation}\label{eq:defi new Psi}
\Psi(z,H,h,\phi)=\begin{cases}
\left(f_{\tilde\phi}(z),h,\phi\right) & \textrm{if }t=0 \\
\left(\Eps_{(f_{\tilde\phi},e^{2\mathsf u(H(t),h,\phi)})},h,\phi\right) & \textrm{if }t>0
\end{cases}
\end{equation}
The map $\Psi$ therefore  induces a continuous map
$$\widehat\Psi:\Sigma\times [0,\delta)\times U_0\to \left(M\cup\partial_\infty^+M\right)\times U_0~.$$
The first step consists in showing that the differential of $\Psi$ (and therefore of $\widehat \Psi)$ is invertible at the points $(z,t=0)$.

\begin{lemma}
For every $z\in\D$ and every pair $(h,\phi)\in U_0$, the differential at $(z,0,h,\phi)$ of the map $\Psi:\D\times [0,\delta)\times U_0\to \left(\Hyp^3\cup\partial_\infty\Hyp^3\right)\times U_0$ defined in \eqref{eq:defi new Psi} is invertible.
\end{lemma}
\begin{proof}
We clearly have that the differential of $\Psi$ is of the form
\[
	d\Psi_{(z,0,h_0,\phi_0)}=\left(\begin{array}{c|c}
		d\Psi_{(z,0)}(\cdot,\cdot,h_0,\phi_0)
		 & \star \\
		 \hline
		0 & \mathbbm 1
	\end{array}\right)
	\]
Hence it suffices to check that the differential of $\Psi(\cdot,\cdot,h_0,\phi_0)$ is invertible, namely, to compute the derivatives with respect to $z$ and $t$ keeping $h$ and $\phi$ fixed.	For this, we can actually reduce to the computation we performed to obtain \eqref{eq:invertible matrix 1}. Indeed, since $f_{\tilde\phi_0}$ is a locally injective holomorphic function, we can change variables from $z$ to $w:=f_{\tilde\phi_0}(z)$ in a small open set on which $f_{\tilde\phi_0}$ is a biholomorphism onto its image. We can then consider $u$, $v$, $\eta$ and $\lambda$ as functions of $w$ instead of $z$, up to composing with a local inverse of $f_{\tilde\phi_0}$. (Of course here $\mathsf u$ and $\mathsf v$ are functions not only of $(z,H)$ but also of $(h,\phi)$, but since we are differentiating with $(h,\phi)$ fixed, the result will remain exactly the same.)

We then obtain, as in \eqref{eq:long}, 
$$\Eps_{(f_{\tilde\phi_0},\sigma(\mathsf u(H(t),h_0,\phi_0))}(w)=(w,0)+ \frac{2}{e^{2\lambda}+\frac{4t^2}{2-t^2}|\lambda_{w}|^{2}}\left(\frac{2t^2}{2-t^2}\lambda_{\bar w},\sqrt{\frac{t^2}{2-t^2}}e^{\lambda}\right)~.$$

Differentiating as above, we obtain the same expression as in \eqref{eq:invertible matrix 1}, which is invertible. Since $w$ is a local coordinate and the choice of $(h_0,\phi_0)$ is arbitrary, the differential of $\Psi$ is invertible at the point $(z,0,h,\phi)$ for any $z,h,\phi$.
\end{proof}
	
Therefore, $\widehat \Psi$ is a local diffeomorphism in a neighbourhood of every point $(z,0,h,\phi)$. We now prove an easy topological lemma.

\begin{lemma}\label{topologylemma}
	Let $X$ be a metrizable compact topological space, $Y$ any topological space and $V$ an open subset of $\mathbb{R}^{n}$ containing the origin. Let $F: X \times V \rightarrow Y$ be a continuous map such that\begin{itemize} 
		\item  $F|_{X\times \left\lbrace 0\right\rbrace }$ is injective and 
		\item $F$ is locally injective at any $(x,0)\in X\times\{0\}$. \end{itemize}Then there exists a neighbourhood $V' \subset V $ of the origin such that $F|_{X\times V'}$ is injective.
\end{lemma}
\begin{proof}
	Assume that there exists no such neighbourhood $V'$ where $F|_{X\times V'}$ is injective. Then there exist sequences $(x_{n},t_{n})_{n\in\mathbb N}$ and $(x'_{n},t'_{n})_{n\in\mathbb N}$ with $t_{n},t'_{n}\rightarrow 0$ such that $(x_n,t_n)\neq (x_n',t_n')$ and $F(x_{n},t_{n})=F(x'_{n},t'_{n})$. Since $X$ is metrizable and compact, it is sequentially compact, and we can extract a convergent subsequence from both $(x_{n})_{n\in\mathbb N}$ and $(x'_{n})_{n\in\mathbb N}$. Let the respective limit points be $x_{\infty}$ and $x'_{\infty}$.  By continuity of $F$ we have that $F(x_{\infty},0)=F(x'_{\infty},0)$ which  implies that $x_{\infty}=x'_{\infty}$ since $F|_{X\times \{0\}}$ is injective. But $F$ is assumed to be locally injective in a neighbourhood of $(x_{\infty},0)$, which means that for $n$ large enough,  $(x_{n},t_{n})=(x'_{n},t'_{n})$. This gives a contradiction. 
\end{proof}

We are now ready to conclude the proof of Theorem \ref{thm:foliation_ends}. Indeed by Lemma \ref{topologylemma} the map $\widehat \Psi$ is an injective local diffeomorphism, if we restrict further its domain of definition, choosing smaller $\delta$ and $U_0$. Hence it is a diffeomorphism onto its image. In particular, composing with the projection to the first factor $M\cup\partial_\infty^+M$ gives a diffeomorphism from $\Sigma\times[0,\delta)$ to its image for all $(h,\phi)$ in $U_0.$ 
Since $H(t)=-1+t^2$ is a diffeomorphism between $(0,\delta)$ and $(-1,-1+\epsilon)$ for $\epsilon=-1+\delta^2$, we have that for every $(h,\phi)$ in $U_0$ and every $H\in(-1,-1+\epsilon)$ the Epstein maps $\Eps_{(f_{\tilde\phi},\sigma(\mathsf u(H,h,\phi))}$ induce a smooth family of embeddings in the quasi-Fuchsian manifold $M$ corresponding to $(h,\phi)$ of constant mean curvature $H$. 

Of course, the same argument can be repeated for $H$ close to $-1$, obtaining a monotone CMC foliation of a neighbourhood of $\partial_\infty^-M$. Clearly, up to choosing a smaller $\epsilon$ and a smaller $U_0$, we can assume that the regions of $m\in U_0$ foliated by surfaces with CMC in $(-1,-1+\epsilon)$ and in $(1-\epsilon,1)$ are disjoint. This means that for every $m\in U_0$, these CMC surfaces foliate the complement of a compact set homeomorphic to $\Sigma\times I$. 
This concludes Theorem \ref{thm:foliation_ends}.

\subsection{Existence in the compact part}\label{sec:existence_compact}

We now prove the existence of CMC surfaces, with mean curvature in $(-1,1)$, in a neighbourhood of any \emph{Fuchsian} manifold. Again, we will need to have some (although very weak) local uniform control on the value of the mean curvature, as in the following statement.

\begin{theorem}\label{thm:existence_compact}
	Let $\Sigma$ be a closed oriented surface of genus $\geq 2$, $H_0\in (-1,1)$ and $m\in\F(\Sigma)$. Then there exists a neighbourhood $U_{H_0}$ of $m$ in $\QF(\Sigma)$ and a constant $\epsilon=\epsilon(m,U_{H_0},H_0)$ such that, for every $H\in (H_0-\epsilon,H_0+\epsilon)$, every quasi-Fuchsian manifold in $U_{H_0}$ contains CMC surfaces with mean curvature $H$, which vary smoothly with respect to $H$. Moreover, we can assume that all such CMC surfaces have principal curvatures in $(-1,1)$. 
\end{theorem}

To prove Theorem \ref{thm:existence_compact}, we will use a similar setting as in Section \ref{sec:foliation_ends}. Roughly, the main idea is to use the implicit function theorem in order to deform the solutions to the CMC problem in a Fuchsian manifold, which are given by umbilical surfaces equidistant from the totally geodesic surface, to solutions to the CMC problem in nearby manifolds and for nearby values of the mean curvature.

\begin{proof}
The proof is very similar to Section \ref{subsec:adapt}.  After the change of variables from $(H,h,\phi,u)$ to $(H,h,\phi,v)$, where $v$ is defined in Equation \eqref{eq:change variables}, the equation of constant mean curvature equal to $H$ for the Epstein map $\Eps_{(f_{\tilde\phi},\sigma_{h}(u))}$ is equivalent to Equation \eqref{eq:us2}, which we rewrite here for the sake of convenience:
$$
	G(H,h,\phi,v):= 1-H -2HK(\tau_h(v) ) + (-1-H)\left(K(\tau_h(v))^{2}- 16\|B(\tau_h(v))-\phi/2\|^{2}_{\tau_h(v)}\right)=0~,
$$
for $\tau_h(v)=e^{2v}h$. We consider again $G$ as a map from $\R\times W\times W^{2,s}(\Sigma,h_0)$ to $W^{2,s - 2}(\Sigma,h_0)$, where $h_0$ is some fixed hyperbolic metric on $\Sigma$. One checks directly that, for any $H_0\in (-1,1)$, the point $(H_0,h_0,\phi_0,v_0)$ is a solution, where $v_0\equiv 0$ and $\phi_0\equiv 0$. This uses that $B(h_0)=0$ because $h_0$ lifts to the Poincar\'e metric on $\D$, which is M\"obius flat, as discussed in Section \ref{subsec:mob flat}. Of course this solution corresponds geometrically to the umbilical CMC surface in the Fuchsian manifold, obtained as an equidistant surface from the totally geodesic surface. 

Hence to apply the implicit function theorem for Banach spaces, we differentiate $G$ with respect to $v$. The differential of the term $\|B(\tau_{h}(v))-\phi/2\|^{2}_{\tau_h(v)}$ vanishes because 
$$B(\tau_{h_0}(v_0))-\phi_0/2=0~,$$ for the same reason as above. We therefore have, similarly to the proof of Theorem \ref{thm:foliation_ends} (see Equation \eqref{eq:differential}):

\begin{align*}
	d_{v}G_{(H_0,h_0,\phi_0,v_0)} &= (-2H_0+1+H_0) \left.\frac{d}{dv}\right|_{v=v_0} (K(e^{2v}h)) \\
	&= (1-H_0)(2\dot v-\Delta_{h_0}\dot v)
\end{align*} 
Since $H_0\neq 1$, $d_{v}G_{(H_0,h_0,\phi_0,v_0)}$ is invertible by Lemma \ref{lem filippo}, and we therefore obtain a family $\mathsf v:[-1,1+\epsilon)\times U_0\to W^{2,s}(\Sigma,h_0)$ of smooth solutions, depending smoothly on $H$. 

Define $\mathsf u:[-1,1+\epsilon)\times U_0\to W^{2,s}(\Sigma,h_0)$ as in \eqref{eq solutions u}. We claim that the Epstein map $\Eps_{(\tilde f_{\phi_0},e^{2\mathsf u(H_0,h_0,\phi_0)}h_0)}=\Eps_{(\mathrm{id},e^{2u_0}h_0)}$, where 
$$u_0=\mathsf u(H_0,h_0,\phi_0)=-\frac{1}{2}\log\frac{1+H_0}{1-H_0}~,$$ 
is an immersion with first fundamental form equal to a multiple of the hyperbolic metric $h_0$. This is of course what we expect since the geometric meaning of the solution $(H_0,h_0,\phi_0,v_0)$ is the umbilical CMC surface that descends to an equidistant surface from the totally geodesic surface in the Fuchsian manifold. The claim can actually be checked without any computation, because the Poincar\'e metric on $\D$, the vanishing quadratic differential $\phi_0$ and the constant function $u_0$ are all invariant under the group of biholomorphisms of $\D$. Hence one can use the uniqueness property in Proposition \ref{prop:eps map} to deduce that there exists a surface $S$ in $\Hyp^3$, equidistant from the totally geodesic plane whose boundary  coincides with $\partial\D$, such that Epstein map $\Eps_{(\mathrm{id},e^{2u_0}h_0)}$ is the unique embedding $\iota:\D\to S\subset\Hyp^3$ satisfying 
$$\iota\circ \zeta=\zeta\circ\iota$$
for every biholomorphism $\zeta$ of $\D$.

Since being an immersion is an open condition, up to restricting the neighbourhood $U_{H_0}$ and taking a smaller $\epsilon$, we can therefore assume that all Epstein maps
$$\Eps_{(f_{\tilde\phi},e^{2\mathsf u(H,h,\phi)})}:\D\to\Hyp^3$$
 are immersions, which have constant mean curvature equal to $H$ by construction.
Hence these Epstein maps induce CMC surfaces in the quotient quasi-Fuchsian manifolds corresponding to the points $(h,\phi)$ in a neighbourhood of $(h_0,\phi_0)$.

The ``moreover'' part of the statement follows again by continuity, up to restricting the neighbourhood $U_{H_0}$ and taking a smaller $\epsilon$, since the principal curvatures of the umbilical CMC surface with mean curvature $H$ are both equal to $H$, and therefore smaller than one in absolute value.
\end{proof}

	\subsection{Conclusion of existence in a small neighbourhood}\label{sec:existence}
	
	Based on Theorems \ref{thm:foliation_ends} and \ref{thm:existence_compact}, we are now ready to prove the existence of CMC surfaces for each value of the mean curvature in $(-1,1)$, in a suitable neighbourhood of the Fuchsian locus.
	
	\begin{theorem}\label{thm:existence}
		Let $\Sigma$ be a closed oriented surface of genus $\geq 2$. Then there exists a neighbourhood $U$ of the Fuchsian locus in quasi-Fuchsian space $\QF(\Sigma)$ such that, for every $H\in (-1,1)$, every quasi-Fuchsian manifold in $U$ contains an embedded CMC surface of mean curvature $H$.
	\end{theorem}
	\begin{proof}
		We will show that, for every $m\in\F(\Sigma)$, there exists a neighbourhood $V=V(m)$ of $m$ in $\QF(\Sigma)$ such that every $m'$ in $V$ contains embedded CMC surfaces for all $H\in (-1,1)$. Taking the union of $V(m)$ as   $m$ varies in $\F(\Sigma)$ clearly provides the claimed neighbourhood of the Fuchsian locus.
		
		Let us fix a convenient notation. For the sake of simplicity, we fix $m$ in $\F(\Sigma)$, and we will omit every dependence on $m$. Theorems \ref{thm:foliation_ends} and \ref{thm:existence_compact} provide us with:
		\begin{enumerate}
			\item A neighbourhood $\widehat U$ of $m$ and a constant $\widehat \epsilon$ such that all quasi-Fuchsian manifolds in $\widehat U$ contain embedded CMC surfaces with mean curvature $H$ ranging in $(-1,-1+\widehat\epsilon)\cup (1-\widehat\epsilon,1)$, and
			\item For every $H_0\in(-1,1)$, a neighbourhood $U_{H_0}$ of $m$ and a constant $\epsilon_{H_0}$ such that   all quasi-Fuchsian manifolds in $\widehat U$ contain immersed CMC surfaces with mean curvature $H$ ranging in $(H_0-\epsilon_{H_0},H_0+\epsilon_{H_0})$ (clearly, $\epsilon_{H_0}$ will be small enough so that $(H_0-\epsilon_{H_0},H_0+\epsilon_{H_0})\subset (-1,1)$).
		\end{enumerate}
		
Actually, in item (2), we can assume that the immersed CMC surfaces have  principal curvatures in $(-1,1)$. This implies automatically that they are embedded, see item \ref{item1} of Proposition \ref{prop:small} below.
		
		Now, the family of intervals 
		$$\mathcal F:=\left\{[-1,-1+\widehat\epsilon)\right\}\cup \left\{(1-\widehat\epsilon,1]\right\}\cup\{(H_0-\epsilon_{H_0},H_0+\epsilon_{H_0})\,|\,H_0\in(-1,1)\}$$
		is an open covering of the compact interval $[-1,1]$, hence it admits a finite subcover
		$$\mathcal F':=\left\{[-1,-1+\widehat\epsilon)\right\}\cup \left\{(1-\widehat\epsilon,1]\right\}\cup\{(H_0-\epsilon_{H_0},H_0+\epsilon_{H_0})\,|\,H_0\in\{c_1,\ldots,c_N\}\}~.$$
		Therefore the intersection
		$$U:=\widehat U\cap U_{c_1}\cap\ldots\cap U_{c_N}$$
		is an open neighbourhood of $m$ in $\QF(\Sigma)$ with the property that for every $H\in (-1,1)$ and for every $m'$ in $U$ there exists an embedded CMC surface with constant mean curvature $H$. This concludes the proof.
	\end{proof}
	
	In the next section, we will improve the proof of Theorem \ref{thm:existence} in order to prove that the neighbourhood $U$ can be taken so as to have the property that the embedded CMC surfaces of each quasi-Fuchsian manifold $M$ in $U$ constitute a smooth monotone foliation of $M$.

\section{Foliations of quasi-Fuchsian manifolds}\label{sec:finish}

Having established the existence of embedded CMC surfaces, for $H\in (-1,1)$, in a quasi-Fuchsian manifold in a suitably small neighbourhood of the Fuchsian locus, we now refine the construction to show that, in a possibly smaller neighbourhood, there is a monotone smooth foliation by CMC surfaces.

\subsection{Small principal curvatures and equidistant foliations}

We will say that a $C^2$ immersion of a surface in $\Hyp^3$ has small principal curvatures if its principal curvatures are in $(-1,1)$. The following statement contains the fundamental properties that we will use on surfaces with small principal curvatures.

\begin{prop}\label{prop:small}
Let $\Sigma$ be a closed surface and let $\iota:\Sigma\to M$ be an immersion with small principal curvatures in a quasi-Fuchsian manifold $M$ homeomorphic to $\Sigma\times\R$. Then:
\begin{enumerate}
\myitem{$i)$}  \label{item1} The immersion $\iota$ is an embedding and a homotopy equivalence.
\myitem{$ii)$} \label{item2}There is a diffeomorphism $\zeta:\Sigma\times\R\to M$ such that $\zeta(\cdot,0)=\iota$, $\zeta(p,\cdot)$ is the unit speed geodesic intersecting $\iota(\Sigma)$ orthogonally at $\iota(p)$, and 
\begin{equation}\label{eq:distance}
d_M(\zeta(p,r_1),\zeta(p,r_2))=d_M(\zeta(\Sigma\times\{r_1\}),\zeta(p,r_2))=|r_2-r_1|~.
\end{equation}
\end{enumerate}
Let us choose such $\zeta$ so that $\zeta(\cdot,r)$ approaches $\partial^-_\infty M$ as $r\to+\infty$. If moreover $\iota$ has constant mean curvature $H$, then
\begin{enumerate}[resume]
\myitem{$iii)$} \label{item3} The mean curvature of the surface $\zeta(\Sigma\times\{r\})$ is strictly larger than $H$ if $r>0$ and strictly smaller than $H$ if $r<0$.
\myitem{$iv)$} \label{item4} There exist differentiable functions $f_-,f_+:\R \to \R$ satisfying $f_\pm(0)=H$ and $f'_\pm(r)>0$ for all $r$, such that the mean curvature of $\zeta(\Sigma\times\{r\})$ is between $f_-(r)$ and $f_+(r)$.
\end{enumerate}
\end{prop}
We will refer to the function $r:M\to\R$ as the \emph{signed distance} from the embedded surface $S=\iota(\Sigma)$. 
\begin{proof}
Points  \ref{item1} and  \ref{item2} are well known. For point \ref{item1}, see \cite{epstein} or \cite[Proposition 4.15, Remark 4.22]{elemamseppi}. Let $\widetilde S$ be the lift of $S=\iota(\Sigma)$ to the universal cover $\Hyp^3$. To show point \ref{item2}, the fundamental property is that $\widetilde S$ stays in the concave side of any tangent horosphere (see \cite[Lemma 4.11]{elemamseppi}), hence \emph{a fortiori} on the concave side of any tangent metric ball centered at a point $P$ outside $\widetilde S$. This implies that the geodesics orthogonal to $S$ are pairwise disjoint and form a global foliation in lines of $M$. Moreover, the distance from $S$ is realized along the orthogonal geodesic through $P$. Observe that if $S=\iota(\Sigma)$ has small principal curvatures, then all equidistant surfaces $\zeta(\Sigma\times\{r\})$ also have small principal curvatures (\cite[Chapter 3]{epstein} or \cite[Corollary 4.4]{elemamseppi}). Hence one can repeat the above argument replacing $S$ with $\zeta(\Sigma\times\{r\})$, and conclude \eqref{eq:distance} for all $r_1,r_2$.

To prove points \ref{item3} and \ref{item4}, observe that, with our convention on the mean curvature (see Section \ref{subsec:mean}), the principal curvatures $\lambda_1(r),\lambda_2(r)$ of the embedding $\iota_r:=\zeta(\cdot,r):\Sigma\to M$ at the point $p$ satisfy the formula:
\begin{equation}\label{eq:formula mean r}
\lambda_i(p,r)=\tanh(\mu_i(p)+r)~,
\end{equation}
which is monotone increasing in $r$, where $\lambda_i(p,0)=\tanh\mu_i(p)\in (-1,1)$. Since the mean curvature of $\iota_r$ at $p$ equals $(\lambda_1(p,r)+\lambda_2(p,r))/2$, it follows that it is larger than $H=(\lambda_1(p,0)+\lambda_2(p,0))/2$ if $r>0$ and smaller than $H$ if $r<0$, as claimed in point \ref{item3}.

More precisely, by a direct computation from Equation \eqref{eq:formula mean r} one checks that the derivative of the mean curvature function 
$$r\mapsto H_p(r)=\frac{1}{2}\left( \lambda_1(p,r)+ \lambda_2(p,r)\right)$$ takes value in $(0,1)$ for all $r$. If we fix $r$, using compactness of $\Sigma$ we can define the  functions
$$g_-(r_0)=\min_{p\in S}\left.\frac{d}{dr}\right|_{r=r_0}H_p(r)\qquad g_+(r_0)=\max_{p\in S}\left.\frac{d}{dr}\right|_{r=r_0}H_p(r)~.$$
Integrating $g_-$ and $g_+$, which are both positive everywhere, from $0$ to $r$, one obtains the functions $f_-$ and $f_+$ as in point \ref{item4}. We remark that $g_\pm$ are continuous, hence integrable: indeed, using continuity in $p$ and $r$ of the $r$-derivative of $H_{p}(r)$, we see that if $r_n\to r_\infty$, then a sequence $p_n\in \Sigma$ of minimum points of $(d/dr)H_\bullet(r_n)$ converges up to a subsequence to $p_\infty$, which is necessarily a minimum point of $(d/dr)H_\bullet(r_\infty)$. Hence $g_-(r_\infty)=\lim_n g_-(r_n)$, and analogously for $g_+$ by replacing minimum by maximum.
\end{proof}

\subsection{Maximum principle for CMC surfaces}

In this section we apply Proposition \ref{prop:small} and the geometric maximum principle for mean curvature to achieve two properties which will play a fundamental role in the proof of the foliation result, Theorem \ref{thm:foliation}.

\begin{prop}\label{prop:unique}
Let $M\cong\Sigma\times\R$ be a quasi-Fuchsian manifold and let $S_H$ and $S_H'$ be closed embedded CMC surfaces in $M$ homotopic to $\Sigma\times\{*\}$ with the same mean curvature $H\in (-1,1)$. If $S_H$ has small principal curvatures, then $S_H=S_H'$.
\end{prop}
\begin{proof}
Let $r$ be the signed distance function from $S_H$, given by the diffeomorphism $\zeta$ as in Proposition \ref{prop:small}, applied to the inclusion $\iota$ of $\Sigma$ with image $S_H$. Since $S_H'$ is compact, the restriction of $r$ to $S_H'$ has a maximum $r_{\max}=r(p_{\max})$ and a minimum $r_{\min}=r(p_{\min})$. By Remark \ref{rmk:convention}, the normal vector to $S_H'$ coincides with minus the gradient of the function $r$ at the points $p_{\min}$ and $p_{\max}$.

This implies that $S_H'$ is tangent to the equidistant surface $\zeta(\Sigma\times\{r_{\max}\})$, and entirely contained in the side $\{r\leq r_{\max}\}$, towards which the normal vector is pointing by our convention. By the geometric maximum principle, the mean curvature of $S_H'$, which equals $H$, is larger than the mean curvature of  $\zeta(\Sigma\times\{r_{\max}\})$ at $p_{\max}$. By item \ref{item3} of Proposition \ref{prop:small}, $r_{\max}\leq 0$. Repeating the argument for the minimum point, one obtains $r_{\min}\geq 0$. Hence $r\equiv 0$ on $S_H'$. Since both $S_H$ and $S_H'$ are closed embedded surfaces, they must coincide. 
\end{proof}

Let us now consider the case of two CMC surfaces with different values of the mean curvature.

\begin{lemma}\label{lemma:disjoint}
Let $M\cong\Sigma\times\R$ be a quasi-Fuchsian manifold and let $S_H$ and $S_{H'}$ be closed embedded CMC surfaces in $M$ homotopic to $\Sigma\times\{*\}$, with mean curvature $H$ and $H'$ respectively, for $H\neq H'$. If $S_H$ has small principal curvatures, then $S_H$ and $S_{H'}$ are disjoint, and moreover the signed distance of every point of $S_{H'}$ from $S_H$ is between $f_+^{-1}(H')$ and $f_-^{-1}(H')$, where $f_\pm$ are the increasing functions introduced in Proposition \ref{prop:small}.
\end{lemma}
\begin{proof}
The proof is very similar to Proposition \ref{prop:unique}. Suppose $H'>H$, the other case being analogous. Consider the restriction to $S_{H'}$ of the signed distance function $r$ with respect to $S_H$. This functions admits a minimum $r_{\min}=r(p_{\min})$ and a maximum $r_{\max}=r(p_{\max})$. Hence $S_{H'}$ is tangent to $\zeta(\Sigma\times\{r_{\min}\})$ at $p_{\min}$ and to $\zeta(\Sigma\times\{r_{\max}\})$ at $p_{\max}$, and contained in the region $\{r_{\min}\leq r\leq r_{\max}\}$ between the two. The geometric maximum principle together with item \ref{item4} of Proposition \ref{prop:small} then implies that 
$$f_-(r_{\max})\leq H'\leq f_+(r_{\min})~.$$
This implies that the restriction of $r$ to $S_{H'}$ is at least $r_{\min}\geq f_+^{-1}(H')$, and at most $r_{\max}\leq f_-^{-1}(H')$, as claimed.
\end{proof}

\subsection{Proof of Theorem \ref{thm:foliation}}
Let us now conclude the proof of the smooth monotone foliation result, by putting together all the ingredients. The aim is showing that, for $M$ a quasi-Fuchsian manifold in a suitable neighbourhood $U$ of the Fuchsian locus $\F(S)$, there exists a diffeomorphism between $\Sigma\times (-1,1)$ and $M$ such that, restricted to each slice $\Sigma\times \{H\}$, is an embedding of constant mean curvature $H$. The existence of such CMC surfaces has been proved in Theorem \ref{thm:existence}, so now the goal (up to choosing a smaller neighbourhood $U$ of $\F(\Sigma)$) is achieving the diffeomorphism, thus proving the smooth foliation part.

\begin{proof}[Proof of Theorem \ref{thm:foliation}]
Recall that the proof of Theorem \ref{thm:existence} produces, for every $m$ in $\F(\Sigma)$, a neighbourhood $U$ in $\QF(\Sigma)$ as the intersection 
\begin{equation}\label{eq:finite subcover}
U:=\widehat U\cap U_{c_1}\cap\ldots\cap U_{c_N}~,
\end{equation}
where $\widehat U$ is a neighbourhood of $m$ in which the ends are monotonically foliated by CMC surfaces with mean curvature ranging in $(-1,-1+\widehat\epsilon)\cup (1-\widehat\epsilon,1)$, and the $U_{c_i}$ are neighbourhoods of $m$ obtained from the family $U_{H_0}$ (by extracting a finite cover of the interval $[-1,1]$). Hence for every $i$, in every quasi-Fuchsian manifold inside $U_{c_i}$ we have existence of CMC surfaces of mean curvature ranging in $(c_i-\epsilon_{c_i},c_i+\epsilon_{c_i})$.

Now, let us provide a couple of preliminary observations. First, from Theorem \ref{thm:existence_compact}, we can assume that the $U_{H_0}$ have the property that the CMC surfaces of mean curvature $(H_0-\epsilon_{H_0},H_0+\epsilon_{H_0})$ have small principal curvatures. (In particular, they are embedded by \ref{item1} of Proposition \ref{prop:small}.) Hence in \eqref{eq:finite subcover}, we can assume that all the $U_{c_i}$ have this property. Second, it is harmless to assume that $c_1<\ldots<c_N$ and that the corresponding intervals, namely $(-1,-1+\widehat\epsilon),(c_1-\epsilon_{c_1},c_1+\epsilon_{c_1}), \ldots, (c_N-\epsilon_{c_N},c_N+\epsilon_{c_N}), (1-\widehat\epsilon,1)$ only intersect in pairs (that is, each interval intersects the previous and the next one, and no other), up to choosing smaller $\epsilon$'s.  

Having made these assumptions, using Theorem \ref{thm:existence_compact} we can construct, for any quasi-Fuchsian manifold $M$ in $U$, smooth maps 
$$\xi_{i}:\Sigma\times (c_i-\epsilon_{c_i},c_i+\epsilon_{c_i})\to M$$ 
having the property that $\xi_{c_i}(\Sigma\times\{H\})$ is an embedded CMC surface of mean curvature $H$. Similarly in the ends, from Theorem  \ref{thm:foliation_ends} we get smooth maps 
$$\xi_0:\Sigma\times (-1,-1+\widehat\epsilon)\to M\qquad\text{and}\qquad \xi_{N+1}:\Sigma\times (1-\widehat\epsilon,1)\to M$$
satisfying the analogous property.

By our previous assumption, all the $\xi_{i}(\Sigma\times\{H\})$ have small principal curvatures, if $i\in\{1,\ldots,N\}$. Hence by Proposition \ref{prop:unique}, we have $\xi_{i}(\Sigma\times\{H\})=\xi_{i'}(\Sigma\times\{H\})$ for every $i,i'\in\{0,\ldots,N+1\}$. Using our other assumption, namely that only consecutive intervals overlap, we can iteratively precompose each $\xi_i$, starting from $\xi_1$, with smooth diffeomorphisms of the source that preserve each slice $\Sigma\times\{H\}$, so that $\xi_i(\cdot,H)=\xi_{i+1}(\cdot,H)$ as long as $H$ is in the intersection of the corresponding intervals. Hence we can glue together the $\xi_i$'s to obtain a smooth map 
$$\xi:\Sigma\times(-1,1)\to M$$
such that $\xi(\Sigma\times\{H\})$ is an embedding of a CMC surface with mean curvature $H$, which we denote by $S_H$. 

We claim that $\xi$ is injective. Indeed it is injective on every slice $\Sigma\times\{H\}$, hence it suffices to show that the images of different slices are disjoint. We distinguish three cases. If $H$ is in one of the intervals $(c_i-\epsilon_{c_i},c_i+\epsilon_{c_i})$, then $S_H$ is disjoint by any $S_{H'}$ for $H'\neq H$ by  Lemma \ref{lemma:disjoint}. If $H\in (-1,-1+\widehat\epsilon)$ and $H'\in (1-\widehat\epsilon,1)$, then $S_H$ and $S_{H'}$ are disjoint because the two neighbourhoods of the ends are disjoint. Finally, if both $H$ and $H'$ are in $(-1,-1+\widehat\epsilon)$ or in $(1-\widehat\epsilon,1)$, then $S_H$ and $S_{H'}$ are disjoint by Theorem  \ref{thm:foliation_ends}.

Moreover $\xi$ is surjective by the intermediate value theorem, because it is a diffeomorphism onto a neighbourhood of the ends when restricted to $\Sigma\times (-1,-1+\widehat\epsilon)$ and $\Sigma\times (1-\widehat\epsilon,1)$ by Theorem \ref{thm:foliation_ends}. Hence $\xi$ is a homeomorphism. By the inverse function theorem, to prove that it is a diffeomorphism, and thus conclude  Theorem  \ref{thm:foliation}, it suffices to show that its differential is injective at every $(p,H)$ with $H$ in one of the intervals $(c_i-\epsilon_{c_i},c_i+\epsilon_{c_i})$. 

For this purpose, we know already that the differential of $\xi$ is injective when restricted to $T_p \Sigma\subset T_{(p,H)}(\Sigma\times(-1,1))$, and $d\xi(T_p\Sigma)$ is the tangent space to the CMC surface which we will call $S_H$. Hence it suffices to show that $d\xi(\partial/\partial H)$ is a nonzero vector transverse to $d\xi_{(p,H)}(T_p\Sigma)=T_{\xi(p,H)}S_H$. Here we use that $S_H$ has small principal curvatures, and the equidistant foliation from Proposition \ref{prop:small}. Indeed, it is sufficient to show that $d(r\circ\xi)(\partial/\partial H)$ does not vanish, where $r$ is the signed distance from $S_H$ provided by Proposition \ref{prop:small}. But the last part of Lemma \ref{lemma:disjoint} tells us that  $r\circ\xi$ (which is a differentiable function) is larger than the function $f_+^{-1}$, whose derivative is positive. Hence 
$$\left.\frac{d}{dt}\right|_{t=H}(r\circ\xi)(p,t)>0~.$$
This concludes the proof.
\end{proof}

\bibliographystyle{alpha}
\bibliographystyle{ieeetr}
\bibliography{cmcbiblio.bib}

\end{document}